\documentclass[fleqn,12pt,twoside]{article}

\usepackage[headings]{espcrc1}
\readRCS
$Id: espcrc1.tex,v 1.2 2004/02/24 11:22:11 spepping Exp $
\usepackage{graphicx}
\usepackage[figuresright]{rotating}

\newtheorem{theorem}{Theorem}

\newtheorem{corollary}[theorem]{Corollary}

\newenvironment{proof}[1][Proof.]{\begin{trivlist}
\item[\hskip \labelsep {\bfseries #1}]}{\end{trivlist}}

\newenvironment{acknowledgement}[1][Acknowledgement]{\begin{trivlist}
\item[\hskip \labelsep {\bfseries #1}]}{\end{trivlist}}

\newcommand{\AmS}{{\protect\the\textfont2
  A\kern-.1667em\lower.5ex\hbox{M}\kern-.125emS}}

\usepackage{amsmath}
\usepackage{amsfonts}
\title{On sum edge-coloring of regular, bipartite and split graphs}

\author{P.A. Petrosyan\address[MCSD]{Institute for Informatics and Automation Problems,\\
National Academy of Sciences, 0014, Armenia}%
\address{Department of Informatics and Applied Mathematics,\\
Yerevan State University, 0025, Armenia}%
\thanks {email: pet\_petros@\{ipia.sci.am, ysu.am, yahoo.com\}},
        R.R. Kamalian\addressmark[MCSD]%
\address{Department of Applied Mathematics and Informatics,\\
Russian-Armenian State University, 0051, Armenia}%
\thanks{email: rrkamalian@yahoo.com.}}

\runtitle{On sum edge-coloring of regular, bipartite and split
graphs}\runauthor{P.A. Petrosyan, R.R. Kamalian}

\begin{document}

\maketitle

\begin{abstract}
An edge-coloring of a graph $G$ with natural numbers is called a sum
edge-coloring if the colors of edges incident to any vertex of $G$
are distinct and the sum of the colors of the edges of $G$ is
minimum. The edge-chromatic sum of a graph $G$ is the sum of the
colors of edges in a sum edge-coloring of $G$. It is known that the
problem of finding the edge-chromatic sum of an $r$-regular ($r\geq
3$) graph is $NP$-complete. In this paper we give a polynomial time
$\left(1+\frac{2r}{(r+1)^{2}}\right)$-approximation algorithm for
the edge-chromatic sum problem on $r$-regular graphs for $r\geq 3$.
Also, it is known that the problem of finding the edge-chromatic sum
of bipartite graphs with maximum degree $3$ is $NP$-complete. We
show that the problem remains $NP$-complete even for some restricted
class of bipartite graphs with maximum degree $3$. Finally, we give
upper bounds for the edge-chromatic sum of some split graphs. \\

Keywords: edge-coloring, sum edge-coloring, regular graph, bipartite
graph, split graph

\end{abstract}

\section{Introduction}\

We consider finite undirected graphs that do not contain loops or
multiple edges. Let $V(G)$ and $E(G)$ denote sets of vertices and
edges of $G$, respectively. For $S\subseteq V(G)$, let $G[S]$ denote
the subgraph of $G$ induced by $S$, that is, $V(G[S])=S$ and
$E(G[S])$ consists of those edges of $E(G)$ for which both ends are
in $S$. The degree of a vertex $v\in V(G)$ is denoted by $d_{G}(v)$,
the maximum degree of $G$ by $\Delta(G)$, the chromatic number of
$G$ by $\chi(G)$, and the chromatic index of $G$ by
$\chi^{\prime}(G)$. The terms and concepts that we do not define can
be found in \cite{b4,b26}.

A proper vertex-coloring of a graph $G$ is a mapping $\alpha:
V(G)\rightarrow \mathbf{N}$ such that $\alpha(u)\neq \alpha(v)$ for
every $uv\in E(G)$. If $\alpha$ is a proper vertex-coloring of a
graph $G$, then $\Sigma(G,\alpha)$ denotes the sum of the colors of
the vertices of $G$. For a graph $G$, define the vertex-chromatic
sum $\Sigma(G)$ as follows:
$\Sigma(G)=\min_{\alpha}\Sigma(G,\alpha)$, where minimum is taken
among all possible proper vertex-colorings of $G$. If $\alpha$ is a
proper vertex-coloring of a graph $G$ and
$\Sigma(G)=\Sigma(G,\alpha)$, then $\alpha$ is called a sum
vertex-coloring. The strength of a graph $G$ ($s(G)$) is the minimum
number of colors needed for a sum vertex-coloring of $G$. The
concept of sum vertex-coloring and vertex-chromatic sum was
introduced by Kubicka \cite{b16}  and Supowit \cite{b22}. In
\cite{b17}, Kubicka and Schwenk showed that the problem of finding
the vertex-chromatic sum is $NP$-complete in general and polynomial
time solvable for trees. Jansen \cite{b12} gave a dynamic
programming algorithm for partial $k$-trees. In papers
\cite{b5,b6,b10,b13,b18}, some approximation algorithms were given
for various classes of graphs. For the strength of graphs,
Brook's-type theorem was proved in \cite{b11}. On the other hand,
there are graphs with $s(G)>\chi(G)$ \cite{b8}. Some bounds for the
vertex-chromatic sum of a graph were given in \cite{b23}.

Similar to the sum vertex-coloring and vertex-chromatic sum of
graphs, in \cite{b5,b9,b11}, sum edge-coloring and edge-chromatic
sum of graphs was introduced. A proper edge-coloring of a graph $G$
is a mapping $\alpha: E(G)\rightarrow \mathbf{N}$ such that
$\alpha(e)\neq \alpha(e^{\prime})$ for every pair of adjacent edges
$e,e^{\prime}\in E(G)$. If $\alpha$ is a proper edge-coloring of a
graph $G$, then $\Sigma^{\prime}(G,\alpha)$ denotes the sum of the
colors of the edges of $G$. For a graph $G$, define the
edge-chromatic sum $\Sigma^{\prime}(G)$ as follows:
$\Sigma^{\prime}(G)=\min_{\alpha}\Sigma^{\prime}(G,\alpha)$, where
minimum is taken among all possible proper edge-colorings of $G$. If
$\alpha$ is a proper edge-coloring of a graph $G$ and
$\Sigma^{\prime}(G)=\Sigma^{\prime}(G,\alpha)$, then $\alpha$ is
called a sum edge-coloring. The edge-strength of a graph $G$
($s^{\prime}(G)$) is the minimum number of colors needed for a sum
edge-coloring of $G$. For the edge-strength of graphs, Vizing's-type
theorem was proved in \cite{b11}. In \cite{b5}, Bar-Noy et al.
proved that the problem of finding the edge-chromatic sum is
$NP$-hard for multigraphs. Later, in \cite{b9}, it was shown that
the problem is $NP$-complete for bipartite graphs with maximum
degree $3$. Also, in \cite{b9}, the authors proved that the problem
can be solved in polynomial time for trees and that
$s^{\prime}(G)=\chi^{\prime}(G)$ for bipartite graphs. In
\cite{b20}, Salavatipour proved that determining the edge-chromatic
sum and the edge-strength are $NP$-complete for $r$-regular graphs
with $r\geq 3$. Also he proved that $s^{\prime}(G)=\chi^{\prime}(G)$
for regular graphs. On the other hand, there are graphs with
$\chi^{\prime}(G)=\Delta(G)$ and $s^{\prime}(G)=\Delta(G)+1$
\cite{b11}. Recently, Cardinal et al. \cite{b7} determined the
edge-strength of the multicycles.

In the present paper we give a polynomial time
$\frac{11}{8}$-approximation algorithm for the edge-chromatic sum
problem of $r$-regular graphs for $r\geq 3$. Next, we show that the
problem of finding the edge-chromatic sum remains $NP$-complete even
for some restricted class of bipartite graphs with maximum degree
$3$. Finally, we give upper bounds for the edge-chromatic sum of
some split graphs.
\bigskip

\section{Definitions and necessary results}\

A proper $t$-coloring is a proper edge-coloring which makes use of
$t$ different colors. If $\alpha $ is a proper $t$-coloring of $G$
and $v\in V(G)$, then $S\left(v,\alpha\right)$ denotes set of colors
appearing on edges incident to $v$. Let $G$ be a graph and
$R\subseteq V(G)$. A proper $t$-coloring of a graph $G$ is called an
$R$-sequential $t$-coloring \cite{b1,b2} if the edges incident to
each vertex $v\in R$ are colored by the colors $1,\ldots,d_{G}(v)$.
For positive integers $a$ and $b$, we denote by $\left[a,b\right]$,
the set of all positive integers $c$ with $a\leq c\leq b$. For a
positive integer $n$, let $K_{n}$ denote the complete graph on $n$
vertices.

We will use the following four results.

\begin{theorem}
\label{mytheorem0} \cite{b15}. If $G$ is a bipartite graph, then
$\chi^{\prime }(G)=\Delta(G)$.
\end{theorem}

\begin{theorem}
\label{mytheorem1} \cite{b24}. For every graph $G$,
\begin{center}
$\Delta(G)\leq \chi^{\prime }(G)\leq \Delta(G)+1$.
\end{center}
\end{theorem}

\begin{theorem}
\label{mytheorem2}\cite{b25}. For the complete graph $K_{n}$ with
$n\geq 2$,
\begin{center}
$\chi^{\prime}(K_{n})=\left\{
\begin{tabular}{ll}
$n-1$, & if $n$ is even, \\
$n$, & if $n$ is odd. \\
\end{tabular}%
\right.$
\end{center}
\end{theorem}

\begin{theorem}
\label{mytheorem3} \cite{b9,b11}. If $G$ is a bipartite or a regular
graph, then $s^{\prime}(G)=\chi^{\prime}(G)$.
\end{theorem}
\bigskip

\section{Edge-chromatic sums of regular graphs}\

In this section we consider the problem of finding the
edge-chromatic sum of regular graphs. It is easy to show that the
edge-chromatic sum problem of graphs $G$ with $\Delta(G)\leq 2$ can
be solved in polynomial time. On the other hand, in \cite{b19}, it
was proved that the problem of finding the edge-chromatic sum of an
$r$-regular ($r\geq 3$) graph is $NP$-complete. Clearly,
$\Sigma^{\prime}(G)\geq \frac{nr(r+1)}{4}$ for any $r$-regular graph
$G$ with $n$ vertices, since the sum of colors appearing on the
edges incident to any vertex is at least $\frac{r(r+1)}{2}$.
Moreover, it is easy to see that
$\Sigma^{\prime}(G)=\frac{nr(r+1)}{4}$ if and only if
$\chi^{\prime}(G)=r$ for any $r$-regular graph $G$ with $n$
vertices.\\

First we give a result on $R$-sequential colorings of regular graphs
and then we use this result for constructing an approximation
algorithm.

\begin{theorem}
\label{mytheorem4} If $G$ is an $r$-regular graph with $n$ vertices,
then $G$ has an $R$-sequential $(r+1)$-coloring with $\vert R\vert
\geq \left\lceil\frac{n}{r+1}\right\rceil$.
\end{theorem}
\begin{proof} By Theorem \ref{mytheorem1}, there exists a proper $(r+1)$-coloring
$\alpha$ of the graph $G$. For $i=1,2,\ldots,r+1$, define the set
$V_{\alpha}(i)$ as follows:
\begin{center}
$V_{\alpha}(i)=\left\{v\in V(G):i\notin S(v,\alpha)\right\}$.
\end{center}

Clearly, for any $i^{\prime},i^{\prime\prime}, 1\leq
i^{\prime}<i^{\prime\prime}\leq r+1$, we have

\begin{center}
$V_{\alpha}(i^{\prime})\cap
V_{\alpha}(i^{\prime\prime})=\emptyset$~~
and~~~$\underset{i=1}{\overset{r+1}{\bigcup }}V_{\alpha}(i)=V(G)$.
\end{center}

Hence,

\begin{center}
$n=\vert V(G)\vert =\left\vert \underset{i=1}{\overset{r+1}{\bigcup
}}V_{\alpha}(i)\right\vert=\underset{i=1}{\overset{r+1}{\sum }}\vert
V_{\alpha}(i)\vert$.
\end{center}

This implies that there exists $i_{0}$, $1\leq i_{0}\leq r+1$, for
which $\vert V_{\alpha}(i_{0})\vert \geq
\left\lceil\frac{n}{r+1}\right\rceil$. Let $R=V_{\alpha}(i_{0})$.

If $i_{0}=r+1$, then $\alpha$ is an $R$-sequential $(r+1)$-coloring
of $G$; otherwise define an edge-coloring $\beta$ as follows: for
any $e\in E(G)$, let

\begin{center}
$\beta(e)=\left\{
\begin{tabular}{ll}
$\alpha(e)$, & if $\alpha(e)\neq i_{0},r+1$,\\
$i_{0}$, & if $\alpha(e)=r+1$,\\
$r+1$, & if $\alpha(e)=i_{0}$.\\
\end{tabular}%
\right.$
\end{center}

It is easy to see that $\beta$ is an $R$-sequential $(r+1)$-coloring
of $G$ with $\vert R\vert \geq
\left\lceil\frac{n}{r+1}\right\rceil$. ~$\square$
\end{proof}

\begin{corollary}
\label{mycorollary1} If $G$ is a cubic graph with $n$ vertices, then
$G$ has an $R$-sequential $4$-coloring with $\vert R\vert \geq
\left\lceil\frac{n}{4}\right\rceil$.
\end{corollary}

Note that if $n$ is odd, then the lower bound in Theorem
\ref{mytheorem4} cannot be improved, since the complete graph
$K_{n}$ has an $R$-sequential $n$-coloring with $\vert R\vert =1$.\\

The theorem we are going to prove will be used in section 5.

\begin{theorem}
\label{mytheorem5} For any $n\in \mathbf{N}$, we have
\begin{center}
$\Sigma^{\prime}(K_{n})=\left\{
\begin{tabular}{ll}
$\frac {n(n^{2}-1)}{4}$, & if $n$ is odd,\\
$\frac {(n-1)n^{2}}{4}$, & if $n$ is even.\\
\end{tabular}%
\right.$
\end{center}
\end{theorem}
\begin{proof} Since for any $r$-regular graph $G$ with $n$ vertices, $\Sigma^{\prime}(G)=\frac{nr(r+1)}{4}$
if and only if $\chi^{\prime}(G)=r$ and, by Theorems
\ref{mytheorem2} and \ref{mytheorem3}, we obtain
$\Sigma^{\prime}(K_{n})=\frac{(n-1)n^{2}}{4}$ if $n$ is even.

Now let $n$ be an odd number and $n\geq 3$. In this case by Theorems
\ref{mytheorem2} and \ref{mytheorem3}, we have
$s^{\prime}(K_{n})=\chi^{\prime}(K_{n})=n$. It is easy to see that
in any proper $n$-coloring of $K_{n}$ the missing colors at $n$
vertices are all distinct. Hence,
\begin{eqnarray*}
\Sigma^{\prime}(K_{n})=\frac{\frac{n^{2}(n+1)}{2}-\frac{n(n+1)}{2}}{2}=\frac
{n(n^{2}-1)}{4}.
\end{eqnarray*} ~$\square$
\end{proof}

In \cite{b5}, it was shown that there exists a $2$-approximation
algorithm for the edge-chromatic sum problem on general graphs. Now
we show that there exists a
$\left(1+\frac{2r}{(r+1)^{2}}\right)$-approximation algorithm for
the edge-chromatic sum problem on $r$-regular graphs for $r\geq 3$.
Note that $1+\frac{2r}{(r+1)^{2}}$ decreases for increasing $r$ and
$\frac{11}{8}$ is its maximum value achieved for $r=3$. Thus, we
show that there is a $\frac{11}{8}$-approximation algorithm for the
edge-chromatic sum problem on regular graphs.

\begin{theorem}
\label{mytheorem6} For any $r\geq 3$, there is a polynomial time
$\left(1+\frac{2r}{(r+1)^{2}}\right)$-approximation algorithm for
the edge-chromatic sum problem on $r$-regular graphs.
\end{theorem}
\begin{proof} Let $G$ be an $r$-regular graph with $n$ vertices and $m$ edges.
Now we describe a polynomial time algorithm $A$ for constructing a
special proper $(r+1)$-coloring of $G$. First we construct a proper
$(r+1)$-coloring $\alpha$ of $G$ in $O(mn)$ time \cite{b21}. Next we
recolor some edges as it is described in the proof of Theorem
\ref{mytheorem4} to obtain an $R$-sequential $(r+1)$-coloring
$\beta$ of $G$ with $\vert R\vert \geq
\left\lceil\frac{n}{r+1}\right\rceil$. Clearly, we can do it in
$O(m)$ time. Now, taking into account that the sum of colors
appearing on the edges incident to any vertex is at most
$\frac{r(r+3)}{2}$, we have

\begin{eqnarray*}
\Sigma_{A}^{\prime}(G)=\Sigma^{\prime}\left(G,\beta\right) &\leq&
\frac{\frac{r(r+1)}{2}\left\lceil\frac{n}{r+1}\right\rceil+\left(n-\left\lceil\frac{n}{r+1}\right\rceil\right)\frac{r(r+3)}{2}}{2}\leq
\frac{\frac{r(r+1)}{2}\frac{n}{r+1}+\left(n-\frac{n}{r+1}\right)\frac{r(r+3)}{2}}{2}\\
&=&
\frac{\frac{r(r+1)}{2}\frac{n}{r+1}+\frac{nr}{r+1}\frac{r(r+3)}{2}}{2}=\frac{nr(r^{2}+4r+1)}{4(r+1)}.
\end{eqnarray*}

On the other hand, since $\Sigma^{\prime}(G)\geq \frac{nr(r+1)}{4}$,
we get

\begin{eqnarray*}
\frac{\Sigma_{A}^{\prime}(G)}{\Sigma^{\prime}(G)}\leq
\frac{nr(r^{2}+4r+1)}{4(r+1)}\cdot\frac{4}{nr(r+1)}=\frac{r^{2}+4r+1}{(r+1)^{2}}=1+\frac{2r}{(r+1)^{2}}.
\end{eqnarray*}

This shows that there exists a
$\left(1+\frac{2r}{(r+1)^{2}}\right)$-approximation algorithm for
the edge-chromatic sum problem on $r$-regular graphs. Moreover, we
can construct the aforementioned coloring $\beta$ for a regular
graph in $O(mn)$ time. ~$\square$
\end{proof}\

\section{Edge-chromatic sums of bipartite graphs}\

In this section we consider the problem of finding the
edge-chromatic sum of bipartite graphs. Let $G=(U\cup W,E)$ be a
bipartite graph with a bipartition $(U,W)$. By $U_{i}\subseteq U$
and $W_{i}\subseteq W$, we denote sets of vertices of degree $i$ in
$U$ and $W$, respectively. Define sets $V_{\geq i}\subseteq V(G)$
and $U_{\geq i}\subseteq U$ as follows: $V_{\geq i}=\{v:v\in
V(G)\wedge d_{G}(v)\geq i\}$ and $U_{\geq i}=\{u\in V(G):u\in
U\wedge d_{G}(u)\geq i\}$. It was proved the following:

\begin{theorem}\label{mytheorem7} \cite{b1,b2,b3,b4} If $G=(U\cup W,E)$ is a bipartite graph with
$d_{G}(u)\geq d_{G}(w)$ for every $uw\in E(G)$, where $u\in U$ and
$w\in W$, then $G$ has a $U$-sequential $\Delta(G)$-coloring.
\end{theorem}

By this theorem, we obtain the following corollary:

\begin{corollary}
\label{mycorollary3} If $G=(U\cup W,E)$ is a bipartite graph with
$d_{G}(u)\geq d_{G}(w)$ for every $uw\in E(G)$, where $u\in U$ and
$w\in W$, then a $U$-sequential $\Delta(G)$-coloring of $G$ is a sum
edge-coloring of $G$ and $\Sigma^{\prime}(G)=\sum_{u\in
U}\frac{d_{G}(u)(d_{G}(u)+1)}{2}$.
\end{corollary}

In \cite{b9}, it was shown that the problem of finding the
edge-chromatic sum of bipartite graphs $G$ with $\Delta(G)=3$ is
$NP$-complete. Now we give a short proof of this fact. First we need
the following

Problem 1. \cite{b2,b4,b14}

Instance: A bipartite graph $G=(U\cup W,E)$ with $\Delta(G)=3$.

Question: Is there a $U$-sequential $3$-coloring of $G$?

It was proved the following:

\begin{theorem}
\label{mytheorem8.0}\cite{b2,b14} Problem 1 is $NP$-complete.
\end{theorem}

Now let us consider the following

Problem 2.

Instance: A bipartite graph $G=(U\cup W,E)$ with $\Delta(G)=3$.

Question: Is $\Sigma^{\prime}(G)=\underset{i=1}{\overset{3}{\sum
}}i\cdot\left\vert U_{\geq i}\right\vert$?

\begin{theorem}
\label{mytheorem8} Problem 2 is $NP$-complete.
\end{theorem}
\begin{proof} Clearly, Problem 2 belongs to $NP$. For the proof
of the $NP$-completeness, we show a reduction from Problem 1 to
Problem 2. We prove that a bipartite graph $G=(U\cup W,E)$ with
$\Delta(G)=3$ admits a $U$-sequential $3$-coloring if and only if
$\Sigma^{\prime}(G)=\underset{i=1}{\overset{3}{\sum
}}i\cdot\left\vert U_{\geq i}\right\vert$. Let $G=(U\cup W,E)$ be a
bipartite graph with $\Delta(G)=3$ and $\alpha$ be a $U$-sequential
$3$-coloring of $G$. In this case the colors $1,2,3$ appear on the
edges incident to each vertex $u\in U_{3}$, the colors $1,2$ appear
on the edges incident to each vertex $u\in U_{2}$ and the color $1$
appears on the pendant edges incident to each vertex $u\in U_{1}$.
Hence, $\Sigma^{\prime}(G,\alpha)=\underset{i=1}{\overset{3}{\sum
}}i\cdot\left\vert U_{\geq i}\right\vert$. On the other hand,
clearly, $\Sigma^{\prime}(G)\geq \underset{i=1}{\overset{3}{\sum
}}i\cdot\left\vert U_{\geq i}\right\vert$, thus
$\Sigma^{\prime}(G)=\underset{i=1}{\overset{3}{\sum
}}i\cdot\left\vert U_{\geq i}\right\vert$.

Now suppose that $\Sigma^{\prime}(G)=\underset{i=1}{\overset{3}{\sum
}}i\cdot\left\vert U_{\geq i}\right\vert$. By Theorems
\ref{mytheorem0} and \ref{mytheorem3}, there exists a proper
$3$-coloring $\beta$ of a bipartite graph $G$ with $\Delta(G)=3$.
This implies that the colors $1,2,3$ appear on the edges incident to
each vertex $u\in U_{3}$. If the color $3$ appears on the edges
incident to some vertices $u\in U_{2}$ or the colors $2$ or $3$
appear on the pendant edges incident to some vertices $u\in U_{1}$,
then it is easy to see that
$\Sigma^{\prime}(G,\beta)>\underset{i=1}{\overset{3}{\sum
}}i\cdot\left\vert U_{\geq i}\right\vert$. Hence, $\beta$ is a
$U$-sequential $3$-coloring of $G$. ~$\square$
\end{proof}

Now we prove that the problem of finding the edge-chromatic sum of
bipartite graphs $G$ with $\Delta(G)=3$ and with additional
conditions is $NP$-complete, too. We need the following

Problem 3. \cite{b2,b14}

Instance: A bipartite graph $G=(U\cup W,E)$ with $\Delta(G)=3$ and
$\vert U_{i}\vert=\vert W_{i}\vert$ for $i=1,2,3$.

Question: Is there a $V(G)$-sequential $3$-coloring of $G$?

It was proved the following:

\begin{theorem}
\label{mytheorem9.0}\cite{b2,b14} Problem 3 is $NP$-complete.
\end{theorem}

Now let us consider the following

Problem 4.

Instance: A bipartite graph $G=(U\cup W,E)$ with $\Delta(G)=3$ and
$\vert U_{i}\vert=\vert W_{i}\vert$ for $i=1,2,3$.

Question: Is
$\Sigma^{\prime}(G)=\frac{1}{2}\underset{i=1}{\overset{3}{\sum
}}i\cdot\left\vert V_{\geq i}\right\vert$?

\begin{theorem}
\label{mytheorem9} Problem 4 is $NP$-complete.
\end{theorem}
\begin{proof} Clearly, Problem 4 belongs to $NP$. For the proof
of the $NP$-completeness, we show a reduction from Problem 3 to
Problem 4. We prove that a bipartite graph $G=(U\cup W,E)$ with
$\Delta(G)=3$ and $\vert U_{i}\vert=\vert W_{i}\vert$ for $i=1,2,3$,
admits a $V(G)$-sequential $3$-coloring if and only if
$\Sigma^{\prime}(G)=\frac{1}{2}\underset{i=1}{\overset{3}{\sum
}}i\cdot\left\vert V_{\geq i}\right\vert$. Let $\alpha$ be a
$V(G)$-sequential $3$-coloring of $G$. In this case the colors
$1,2,3$ appear on the edges incident to each vertex $v\in V(G)$ with
$d_{G}(v)=3$, the colors $1,2$ appear on the edges incident to each
vertex $v\in V(G)$ with $d_{G}(v)=2$ and the color $1$ appears on
the pendant edges incident to each vertex $v\in V(G)$ with
$d_{G}(v)=1$. Hence,
$\Sigma^{\prime}(G,\alpha)=\frac{1}{2}\underset{i=1}{\overset{3}{\sum
}}i\cdot\left\vert V_{\geq i}\right\vert$. On the other hand,
clearly, $\Sigma^{\prime}(G)\geq
\frac{1}{2}\underset{i=1}{\overset{3}{\sum }}i\cdot\left\vert
V_{\geq i}\right\vert$, thus
$\Sigma^{\prime}(G)=\frac{1}{2}\underset{i=1}{\overset{3}{\sum
}}i\cdot\left\vert V_{\geq i}\right\vert$.

Now suppose that
$\Sigma^{\prime}(G)=\frac{1}{2}\underset{i=1}{\overset{3}{\sum
}}i\cdot\left\vert V_{\geq i}\right\vert$. By Theorems
\ref{mytheorem0} and \ref{mytheorem3}, there exists a proper
$3$-coloring $\beta$ of a bipartite graph $G$ with $\Delta(G)=3$ and
$\vert U_{i}\vert=\vert W_{i}\vert$ for $i=1,2,3$. This implies that
the colors $1,2,3$ appear on the edges incident to each vertex $v\in
V(G)$ with $d_{G}(v)=3$. If the color $3$ appears on the edges
incident to some vertices $v\in V(G)$ with $d_{G}(v)=2$ or the
colors $2$ or $3$ appear on the pendant edges incident to some
vertices $v\in V(G)$ with $d_{G}(v)=1$, then it is easy to see that
$\Sigma^{\prime}(G,\beta)>\frac{1}{2}\underset{i=1}{\overset{3}{\sum
}}i\cdot\left\vert V_{\geq i}\right\vert$. Hence, $\beta$ is a
$V(G)$-sequential $3$-coloring of $G$. ~$\square$
\end{proof}

In \cite{b19}, it was proved that the problem of finding the
edge-chromatic sum of bipartite graphs $G$ with $\Delta(G)=3$
remains $NP$-hard even for planar bipartite graphs.\\

\section{Edge-chromatic sums of split graphs}\

In this section we consider the problem of finding the
edge-chromatic sum of split graphs. A split graph is a graph whose
vertices can be partitioned into a clique $C$ and an independent set
$I$. Let $G=(C\cup I,E)$ be a split graph, where
$C=\{u_{1},u_{2},\ldots,u_{n}\}$ is clique and
$I=\{v_{1},v_{2},\ldots,v_{m}\}$ is independent set. Define a number
$\Delta_{I}$ as follows: $\Delta_{I}=\max_{1\leq j\leq
m}d_{G}(v_{j})$. Define subgraphs $H$ and $H^{\prime}$ of a graph
$G$ as follows:
\begin{center}
$H=(C\cup I, E(G)\setminus E(G[C]))$ and $H^{\prime}=G[C]$.
\end{center}
Clearly, $H$ is a bipartite graph with a bipartition $(C,I)$, and
$d_{H}(u_{i})=d_{G}(u_{i})-n+1$ for $i=1,2,\ldots,n$,
$d_{H}(v_{j})=d_{G}(v_{j})$ for $j=1,2,\ldots,m$.

\begin{theorem}
\label{mytheorem10} Let $G=(C\cup I,E)$ be a split graph, where
$C=\{u_{1},u_{2},\ldots,u_{n}\}$ is clique and
$I=\{v_{1},v_{2},\ldots,v_{m}\}$ is independent set. If
$d_{G}(u_{i})-d_{G}(v_{j})\geq n-1$ for every $u_{i}v_{j}\in E(G)$,
then:
\begin{description}
\item[(1)] if $n$ is even, then
\begin{center}
$\Sigma^{\prime}(G)\leq {\min
}\left\{\sum_{i=1}^{n}\frac{\left(d_{G}(u_{i})-n+1\right)\left(d_{G}(u_{i})-n+2\right)}{2}+\frac{\left(2\Delta(G)-n+2\right)n(n-1)}{4},
\Sigma^{\prime}(K_{n})+\sum_{i=1}^{n}\frac{\left(d_{G}(u_{i})-n+1\right)\left(d_{G}(u_{i})+n\right)}{2}\right\}$;
\end{center}

\item[(2)] if $n$ is odd, then
\begin{center}
$\Sigma^{\prime}(G)\leq {\min}\left\{
\sum_{i=1}^{n}\frac{\left(d_{G}(u_{i})-n+1\right)\left(d_{G}(u_{i})-n+2\right)}{2}+\frac{\left(2\Delta(G)-n+3\right)n(n-1)}{4},
\Sigma^{\prime}(K_{n})+\sum_{i=1}^{n}\frac{\left(d_{G}(u_{i})-n+1\right)\left(d_{G}(u_{i})+n+2\right)}{2}\right\}$.
\end{center}
\end{description}
\end{theorem}
\begin{proof}
For the proof, we are going to construct edge-colorings that
satisfies the specified conditions.

Since $d_{G}(u_{i})-d_{G}(v_{j})\geq n-1$ for every $u_{i}v_{j}\in
E(G)$, we have $d_{H}(u_{i})\geq d_{H}(v_{j})$ for each
$u_{i}v_{j}\in E(H)$. By Theorem \ref{mytheorem7}, there exists a
$C$-sequential $\Delta(H)$-coloring $\alpha$ of the graph $H$ and,
by Corollary \ref{mycorollary3}, we obtain

\begin{center}
$\Sigma^{\prime}(H)=\Sigma^{\prime}(H,\alpha)=\sum_{i=1}^{n}\frac{d_{H}(u_{i})\left(d_{H}(u_{i})+1\right)}{2}$.
\end{center}

Now we consider two cases.

Case 1: $n$ is even.

In this case, by Theorem \ref{mytheorem2}, we have $\chi^{\prime
}(H^{\prime})=n-1$. Let $\beta$ be a proper edge-coloring of a graph
$H^{\prime}$ with colors $\Delta(G)-n+2,\ldots,\Delta(G)$. Clearly,
for each vertex $u_{i}$, $i=1,2,\ldots,n$, the set of colors
appearing on edges incident to $u_{i}$ in $H^{\prime}$ is
$[\Delta(G)-n+2,\Delta(G)]$. Thus, we obtain

\begin{center}
$\Sigma^{\prime}(G)\leq
\Sigma^{\prime}(H)+\frac{(2\Delta(G)-n+2)n(n-1)}{4}$.
\end{center}

On the other hand, let $\beta^{\prime}$ be a proper edge-coloring of
a graph $H^{\prime}$ with colors $1,2,\ldots,n-1$. Clearly, for each
vertex $u_{i}$, $i=1,2,\ldots,n$, the set of colors appearing on
edges incident to $u_{i}$ in $H^{\prime}$ is $[1,n-1]$. Next, we
define an edge-coloring $\gamma$ of the graph $H$ as follows: for
every $e\in E(H)$, let $\gamma(e)=\alpha(e)+n-1$. Thus, we obtain

\begin{center}
$\Sigma^{\prime}(G)\leq
\Sigma^{\prime}(K_{n})+\sum_{i=1}^{n}\frac{\left(d_{G}(u_{i})-n+1\right)\left(d_{G}(u_{i})+n\right)}{2}$.
\end{center}

Case 2: $n$ is odd.

In this case, by Theorem \ref{mytheorem2}, we have $\chi^{\prime
}(H^{\prime})=n$. Let $\beta$ be a proper edge-coloring of a graph
$H^{\prime}$ with colors $\Delta(G)-n+2,\ldots,\Delta(G)+1$. Without
loss of generality, we may assume that for each vertex $u_{i}$,
$i=1,2,\ldots,n$, the set of colors appearing on edges incident to
$u_{i}$ in $H^{\prime}$ is
$[\Delta(G)-n+2,\Delta(G)+1]\setminus\{\Delta(G)-n+1+i\}$. Thus, we
obtain

\begin{center}
$\Sigma^{\prime}(G)\leq
\Sigma^{\prime}(H)+\frac{(2\Delta(G)-n+3)n(n-1)}{4}$.
\end{center}

On the other hand, let $\beta^{\prime}$ be a proper edge-coloring of
a graph $H^{\prime}$ with colors $1,2,\ldots,n$. Without loss of
generality, we may assume that for each vertex $u_{i}$,
$i=1,2,\ldots,n$, the set of colors appearing on edges incident to
$u_{i}$ in $H^{\prime}$ is $[1,n]\setminus\{i\}$. Next, we define an
edge-coloring $\gamma$ of the graph $H$ as follows: for every $e\in
E(H)$, let $\gamma(e)=\alpha(e)+n$. Thus, we obtain

\begin{center}
$\Sigma^{\prime}(G)\leq
\Sigma^{\prime}(K_{n})+\sum_{i=1}^{n}\frac{\left(d_{G}(u_{i})-n+1\right)\left(d_{G}(u_{i})+n+2\right)}{2}$.
\end{center}
~$\square$
\end{proof}

\begin{theorem}
\label{mytheorem11} Let $G=(C\cup I,E)$ be a split graph, where
$C=\{u_{1},u_{2},\ldots,u_{n}\}$ is clique and
$I=\{v_{1},v_{2},\ldots,v_{m}\}$ is independent set. If
$d_{G}(u_{i})-d_{G}(v_{j})\leq n-1$ for every $u_{i}v_{j}\in E(G)$,
then:
\begin{description}
\item[(1)] if $n$ is even, then
\begin{center}
$\Sigma^{\prime}(G)\leq {\min }\left\{
\sum_{j=1}^{m}\frac{d_{G}(v_{j})\left(d_{G}(v_{j})+1\right)}{2}+\frac{\left(2\Delta_{I}+n\right)n(n-1)}{4},
\Sigma^{\prime}(K_{n})+\sum_{j=1}^{m}\frac{d_{G}(v_{j})\left(d_{G}(v_{j})+2n-1\right)}{2}\right\}$;
\end{center}

\item[(2)] if $n$ is odd, then
\begin{center}
$\Sigma^{\prime}(G)\leq {\min }\left\{
\sum_{j=1}^{m}\frac{d_{G}(v_{j})\left(d_{G}(v_{j})+1\right)}{2}+\frac{\left(2\Delta_{I}+n+1\right)n(n-1)}{4},
\Sigma^{\prime}(K_{n})+\sum_{j=1}^{m}\frac{d_{G}(v_{j})\left(d_{G}(v_{j})+2n+1\right)}{2}\right\}$.
\end{center}
\end{description}
\end{theorem}
\begin{proof}
For the proof, we are going to construct edge-colorings that
satisfies the specified conditions.

Since $d_{G}(u_{i})-d_{G}(v_{j})\leq n-1$ for every $u_{i}v_{j}\in
E(G)$, we have $d_{H}(u_{i})\leq d_{H}(v_{j})$ for each
$u_{i}v_{j}\in E(H)$. By Theorem \ref{mytheorem7}, there exists an
$I$-sequential $\Delta_{I}$-coloring $\alpha$ of the graph $H$ and,
by Corollary \ref{mycorollary3}, we obtain

\begin{center}
$\Sigma^{\prime}(H)=\Sigma^{\prime}(H,\alpha)=\sum_{j=1}^{m}\frac{d_{H}(v_{j})\left(d_{H}(v_{j})+1\right)}{2}=\sum_{j=1}^{m}\frac{d_{G}(v_{j})\left(d_{G}(v_{j})+1\right)}{2}$.
\end{center}

Now we consider two cases.

Case 1: $n$ is even.

In this case, by Theorem \ref{mytheorem2}, we have $\chi^{\prime
}(H^{\prime})=n-1$. Let $\beta$ be a proper edge-coloring of a graph
$H^{\prime}$ with colors $\Delta_{I}+1,\ldots,\Delta_{I}+n-1$.
Clearly, for each vertex $u_{i}$, $i=1,2,\ldots,n$, the set of
colors appearing on edges incident to $u_{i}$ in $H^{\prime}$ is
$[\Delta_{I}+1,\Delta_{I}+n-1]$. Thus, we obtain

\begin{center}
$\Sigma^{\prime}(G)\leq
\Sigma^{\prime}(H)+\frac{\left(2\Delta_{I}+n\right)n(n-1)}{4}$.
\end{center}

On the other hand, let $\beta^{\prime}$ be a proper edge-coloring of
a graph $H^{\prime}$ with colors $1,2,\ldots,n-1$. Clearly, for each
vertex $u_{i}$, $i=1,2,\ldots,n$, the set of colors appearing on
edges incident to $u_{i}$ in $H^{\prime}$ is $[1,n-1]$. Next, we
define an edge-coloring $\gamma$ of the graph $H$ as follows: for
every $e\in E(H)$, let $\gamma(e)=\alpha(e)+n-1$. Thus, we obtain

\begin{center}
$\Sigma^{\prime}(G)\leq
\Sigma^{\prime}(K_{n})+\sum_{j=1}^{m}\frac{d_{G}(v_{j})\left(d_{G}(v_{j})+2n-1\right)}{2}$.
\end{center}

Case 2: $n$ is odd.

In this case, by Theorem \ref{mytheorem2}, we have $\chi^{\prime
}(H^{\prime})=n$. Let $\beta$ be a proper edge-coloring of a graph
$H^{\prime}$ with colors $\Delta_{I}+1,\ldots,\Delta_{I}+n$. Without
loss of generality, we may assume that for each vertex $u_{i}$,
$i=1,2,\ldots,n$, the set of colors appearing on edges incident to
$u_{i}$ in $H^{\prime}$ is
$[\Delta_{I}+1,\Delta_{I}+n]\setminus\{\Delta_{I}+i\}$. Thus, we
obtain

\begin{center}
$\Sigma^{\prime}(G)\leq
\Sigma^{\prime}(H)+\frac{\left(2\Delta_{I}+n+1\right)n(n-1)}{4}$.
\end{center}

On the other hand, let $\beta^{\prime}$ be a proper edge-coloring of
a graph $H^{\prime}$ with colors $1,2,\ldots,n$. Without loss of
generality, we may assume that for each vertex $u_{i}$,
$i=1,2,\ldots,n$, the set of colors appearing on edges incident to
$u_{i}$ in $H^{\prime}$ is $[1,n]\setminus\{i\}$. Next, we define an
edge-coloring $\gamma$ of the graph $H$ as follows: for every $e\in
E(H)$, let $\gamma(e)=\alpha(e)+n$. Thus, we obtain

\begin{center}
$\Sigma^{\prime}(G)\leq
\Sigma^{\prime}(K_{n})+\sum_{j=1}^{m}\frac{d_{G}(v_{j})\left(d_{G}(v_{j})+2n+1\right)}{2}$.
\end{center}
~$\square$
\end{proof}

\begin{acknowledgement}
We would like to thank both referees for many useful suggestions.
\end{acknowledgement}

\end{document}